\documentclass{amsart}

\newtheorem{theorem}{Theorem}[section]

\newtheorem{proposition}[theorem]{Proposition}
\theoremstyle{definition}

\newtheorem{corollary}[theorem]{Corollary}
\theoremstyle{remark}
\newtheorem{remark}[theorem]{Remark}

\numberwithin{equation}{section}

\newcommand{\abs}[1]{\lvert#1\rvert}

\begin{document}

\title{Linear characters and Block algebra}

\author{Jiwen Zeng}
\curraddr{ Mathematic School, Xiamen University, Fujian, China.}
\email{jwzeng@xmu.edu.cn}
\thanks{}



\date{}

\dedicatory{School of Mathematics, Xiamen University\\jwzeng@xmu.edu.cn}


\begin{abstract}
This paper will prove that: 1. $G$ has a block only having linear ordinary characters if and only if $G$ is a $p$-nilpotent group with an abelian Sylow $p$-subgroup; 2. $G$ has a block only having linear Brauer characters if and only if $O_{p'}(G)\leq O_{p'p}(G)=HO_{p'}(G)= \textrm{Ker}(B_{0}^{*}) \leq O_{p'pp'}=G$, where $H=G^{'}O^{p'}(G), \textrm{Ker}(B_{0}^{*})=\bigcap_{\lambda \in \textrm{IBr}(B_{0})} \textrm{Ker}(V_{\lambda}), B_{0}$ is the principal block of $G$ and $V_{\lambda}$ is the $F[G]$-module affording the Brauer character $\lambda$; 3. if $G$ satisfies the conditions above, then for any block algebra $B$ of $G$, we have $$ \frac{\textrm{Dim}_{F}(B)}{|D|}= \sum_{\varphi \in \textrm{IBr}(B)}\varphi(1)^{2}$$ where $D$ is the defect group of $B$.
\end{abstract}

\maketitle
\section*{Introduction}

Let $G$ be a finite group. Let Irr($G$) and IBr($G$) denote the set of irreducible ordinary characters and irreducible Brauer characters of $G$, respectively. We use cd($G$) to denote the set $\{ \chi(1)|\chi \in \textrm{Irr}(G)\}$. The classical result asserts that $G$ has a normal abelian $p$-complement if and only if cd($G$) is a set of $p$-power. J. Thompson\cite{tho} proved that $G$ has a normal $p$-complement if $p$ divides every element $\chi(1)>1$ in cd($G$). G. Navarro and P. H. Tiep \cite{nati}improved this result by considering the fields of character values. Let $B$ be a $p$-block algebra of $G$ and cd$(B)=\{ \chi(1)| \chi \in \textrm{Irr}(B)\}$. M. Isaacs and S.D.Smith\cite{issm} proved that $G$ is $p$-nilpotent  if $p$ divides every element($>1$) in cd($B_{0}$), where $B_{0}$ is the principal block of $G$. G. Navarro and G. Robinson\cite{naro} proved that $B$ is a nilpotent block if  cd($B$) is a set of $p$-power, using the Classification of Finite Simple Groups.

Now let  bcd$(G)=\{ \varphi(1)| \varphi \in \textrm{IBr}(G)\}$.  It is easy to prove that $G$ has a normal Sylow $p$-subgroup $P$ and $G/P$ is abelian if and only if bcd$(G)=\{ 1 \}$. A theorem of Michler\cite{mic} asserts $G$ has a normal Sylow $p-$subgroup if and if only there is no element in bcd($G$) divided by $p$. P. H. Tiep and W. Willems\cite{tiwi} considered another extreme case when  bcd($G$) is a set of a power of a fixed prime number $s$. They proved $G$ is solvable unless $p=s=2$.

For a $p$-block $B$ of $G$, we define bcd($B)=\{  \varphi(1)| \varphi \in \textrm{IBr}(B)\}$. Generally, it is very difficult to get the information of $G$ by using the set bcd$(B)$ or cd$(B)$(see the articles mentioned above). In this paper, we want to decide the group $G$ under the condition of existing some bcd$(B)=\{1 \}$ or cd$(B)=\{1 \}$. We succeed in finding all these groups satisfying our condition( see Theorem 2.4). One typical case is $S_{4}, p=3$. Our ideal mainly comes from two papers, one from G. Chen\cite{chen} and another from X.Y. Chen, etc.\cite{xiao}. G. Chen proved a result on zeros of ordinary characters, and X. Y. Chen, etc. generalized the result to Brauer characters. In this paper, we will improve their results(see Corollary 1.4).

T. Holm and W. Willems\cite{hw} presented some conjectures on Brauer character degrees. One of their conjectures is in block form: let $B$ be a block of $G$ with defect group $D$, is it true the following:
$$\frac{\textrm{Dim}_{F}(B)}{l(B)|D|}\leq \sum_{\varphi \in \textrm{IBr}(B)}\varphi(1)^{2}\;?$$with the equality if and only if $l(B)=1$, where $l(B)=|\textrm{IBr}(B)|$. They have proved the conjecture holds for $p$-solvable groups. Our results show that the groups in our cases are $p$-solvable and furthermore we will get:
$$\frac{\textrm{Dim}_{F}(B)}{|D|}= \sum_{\varphi \in \textrm{IBr}(B)}\varphi(1)^{2}$$ for any $B\in \textrm{Bl}(G)$, if $G$ satisfies our condition.

For most results in Character theory and Brauer character theory,  refer to \cite{is}\cite{nav}\cite{be}\cite{na}.

Section 1 will define an action of linear Brauer characters on IBr($G$) and set of projective characters, giving results on zero points of Brauer characters. Section 2 will define an action of linear Brauer characters on the set of block algebras of $G$, giving our main results in this paper.

The following notation and terminology will be used throughout in this paper. $G$ is a finite group. $p>0$ is a prime number. $(K, R, F)$ is a splitting $p$-module system, that means $R$ is a complete discrete valuation ring with a maximal ideal $\pi$ such that $F=R/\pi$ is a field of characteristic $p$ and $K$ is the quotient  field of $R$. A $p$-regular element means an element of $G$ whose order
is prime to $p$. $G^{0}$ is the set of all $p$-regular elements in $G$. Bl($G$) denotes the set of $p$-blocks of $G$. $B_{0}$ denotes the principal block of $G$, containing the trivial character. We use LBr($G$) to denote the set of all linear Brauer characters of $G$.

\section{Action of ${\rm LBr}(G)$ on ${\rm IBr}(G)$}

Let $G$ be a finite group and let IBr(G) be the set of irreducible Brauer characters of $G$.  For a $\varphi \in  \textrm{IBr}(G)$, We use $V_{\varphi}$ to denote the $F[G]$-module
 affording  the
Brauer character $\varphi$. Let
LBr($G$) denote the set of all linear Brauer characters of $G$. Take $\lambda \in \textrm{LBr}(G)$ and
$\varphi \in \textrm{IBr}(G)$, then $\lambda\varphi$ is a Brauer character afforded
by $V_{\lambda}\otimes V_{\varphi}$, which is also irreducible $F[G]$-module. Hence
$\lambda\varphi \in \textrm{IBr}(G)$. If $V$ is a $F[G]$(right)-module, we define $$
V^{*}=\textrm{Hom}_{F}(V,F)
$$
as $F[G]$-module by $fx(u)=f(ux^{-1}),x\in G,u\in V$, for $f\in \textrm{Hom}_{F}(V,F)$. $V^{*}$ affords
Brauer character $\varphi^{*}$ such that  $\varphi^{*}(x)=\varphi(x^{-1})$, if $V$ affords character $\varphi$.

It follows from discussion above that LBr($G$) is an abelian group and has an action on the set IBr($G$).

Let $P_{\varphi}$ be the projective cover of $\varphi \in \textrm{IBr}(G)$. Since $\lambda\varphi\in \textrm{IBr}(G)$ for $\lambda \in \textrm{LBr}(G)$, what can we say about  $P_{\lambda\varphi}$?
\begin{proposition} For $\lambda \in {\rm LBr}(G),\varphi \in {\rm IBr}(G)$, let $V_{\varphi}$ and $P_{\varphi}$ be the corresponding simple module and projective cover, respectively. Then we have
$P_{\lambda\varphi}=V_{\lambda}\otimes P_{\varphi}.$

\end{proposition}

\begin{proof}
We have $$ P_{\varphi}\longrightarrow V_{\varphi}\longrightarrow 0.
$$
Then $$ V_{\lambda}\otimes P_{\varphi} \longrightarrow V_{\lambda}\otimes V_{\varphi}\longrightarrow 0.
$$
Since $V_{\lambda}\otimes P_{\varphi}$ is projective and indecomposable, it follows that $P_{\lambda\varphi}=V_{\lambda}\otimes P_{\varphi}$ and the proof is complete.
\end{proof}
\begin{remark}Let $\Phi_{\varphi}$ denote the Brauer character afforded by $P_{\varphi}$ for $\varphi \in \textrm{IBr}(G)$. Let $\textrm{PIm}(G)=\{ \Phi_{\varphi}|\varphi \in \textrm{IBr}(G)\}$. By Proposition 1.1,
the group $\textrm{LBr}(G)$ has an action on $\textrm{PIm}(G)$ by $\lambda \Phi_{\varphi}=\Phi_{\lambda\varphi}$.

\end{remark}

Let $\{\Delta_{i}|i=1,2,...,n\}$ be the set of orbits of LBr($G$) acting on IBr($G$). Hence
$$   \textrm{IBr}(G)=\bigcup_{i=1,...,n}\Delta_{i}.$$

Take a sum $$
\sigma=\sum_{\varphi\in \Delta}\varphi,$$
where $\Delta$ is an orbit of LBr($G$) acting on IBr($G$). Then we have $\lambda \sigma=\sigma$,
for all $\lambda\in \textrm{LBr}(G)$. If $g\in G_{p^{\prime}}$ is a $p$-regular element and $\sigma(g)\neq 0$,
 then $$g\in \bigcap_{\lambda \in \textrm{LBr}(G)}\textrm{Ker}(\lambda)=G^{\prime}O^{p^{\prime}}(G),$$
 where $O^{p^{\prime}}(G)$ is a subgroup of $G$ generated by all $p$-elements.

 Notation: $G^{0}$ is the set of $p$-regular elements of $G$. $\triangle_{\varphi}$ is the orbit
 containing $\varphi$ under the action of LBr($G$) on IBr($G$). $\Phi_{\varphi}$ denotes the Brauer character
 of projective cover of $\varphi.$

 By using the discussions above it can be shown the following result:

 \begin{theorem}
 Let $\varphi\in {\rm IBr}(G)$ and $x\in G^{0}-H^{0}$, where $H=G^{\prime}O^{p^{\prime}}(G)$. Then the order $o(xH)||\triangle_{\varphi}|$ or both $\varphi (x)=0$ and $\Phi_{\varphi}(x)=0$.
 \end{theorem}
 \begin{proof} Let $T_{\varphi}$ be the subgroup of LBr($G$) that fixs $\varphi$. Then $\triangle_{\varphi}=\{\lambda_{i}\varphi|\lambda_{i}\in \textrm{LBr}(G)/T_{\varphi}\}$. Let
 $\rho=\prod\lambda_{i}\varphi$ and we denote $|\triangle_{\varphi}|$ by $t$. Hence we have
 $\lambda^{t}\rho=\rho$ for any $\lambda\in \textrm{LBr}(G)$. If $\rho(x)\neq 0$ then $\lambda(x^{t})=1$.
 Hence $x^{t}\in H$ and $o(xH)|t$. If $\rho(x)=0$ then we get $\varphi(x)=0$.

 By Proposition 1.1, $\rho$ has the projective cover with Brauer character $\Phi_{\rho}=\prod\lambda_{i}\Phi_{\varphi}$. Hence $\lambda^{t}\Phi_{\rho}=\Phi_{\rho}.$ It follows
 from the same reason as above that $o(xH)||\triangle_{\varphi}|$ or $\Phi_{\varphi}(x)=0$.

 \end{proof}

 Let $l$ denote the number of irreducible Brauer characters. Since $l=\sum|\triangle_{\varphi}|$ this implies
 the existence of orbit $\triangle_{\varphi}$ such that $o(xH)$ does not divide $|\triangle_{\varphi}|$, if $(l,o(xH))=1.$
 Hence we have the following result as a generalization of \cite{xiao}.

 \begin{corollary} For $x\in G^{0}-H^{0}$. Then we have
 \begin{enumerate}

 \item If $o(xH)$ does not divide $l$, then there exists $\varphi\in {\rm IBr}(G)$ such that $\varphi(x)=0$ and $\Phi_{\varphi}(x)=0$.

\item If  $(o(xH),l)=1$, then there exists $\varphi\in {\rm IBr}(G)$
 such that $\varphi(x)=0$ and $\Phi_{\varphi}(x)=0$.
 \end{enumerate}
 \end{corollary}
 \begin{proof} This is an immediate consequence of Theorem 1.3.
\end{proof}

\section{Action of ${\rm LBr}(G)$ on ${\rm Bl}(G)$}

Bl($G$) denote the set of block algebra of $F[G]$.

For $B\in \textrm{Bl}(G)$ and $\lambda \in \textrm{LBr}(G)$, we define
$\textrm{IBr}(\lambda B)=\{ \lambda\varphi |\varphi \in \textrm{IBr}(B)\}$. Choose any two elements
 $\phi,\varphi\in \textrm{IBr}(B)$. Then
 \begin{equation}
 (\Phi_{\lambda\varphi},\Phi_{\lambda\phi} )=(\Phi_{\varphi},\Phi_{\phi})
 \end{equation}
 and
 \begin{equation}(\lambda\varphi,\lambda\phi)_{G^{0}}=(\varphi,\phi)_{G^{0}}
 \end{equation}
 by which we have that $\lambda\phi\in \textrm{IBr}(\lambda B)$ if and only if
 $\phi \in \textrm{IBr}(B)$. Hence $\lambda B \in \textrm{Bl}(G)$. This defines an action
 of LBr($G$) on Bl($G$). By (2.1) and (2.2), the action keeps inner product of characters. Hence we have

 \begin{proposition}Let $\lambda\in {\rm LBr}(G)$ and $B\in {\rm Bl}(G)$. Then there is a inner-product-preserving bijection from $B$ to $\lambda B$ under the map: $\varphi \mapsto \lambda\varphi, \Phi_{\varphi}\mapsto \lambda \Phi_{\varphi}=\Phi_{\lambda\varphi}$.
In particular, $B$ and $\lambda B$ has the same Cartan matrix.

 \end{proposition}
 Let $B_{0}$ denote the principal block of $G$, namely the unique one which contains trivial character $1.$

 \begin{theorem}Let $B$ be a block of $G$ containing linear character $\lambda \in {\rm LBr}(G)$.
 Then $B$ is isomorphic to $B_{0}$ in the means of Proposition 2.1.
 \end{theorem}
\begin{proof} Since $\lambda \in {\rm LBr}(B)$, we have $B=\lambda B_{0}$. By Proposition 2.1,
the result follows.
\end{proof}

A wellknown fact is that all of irreducible Brauer characters are linear if and only if $G$ has a normal $p$-subgroup $P$ and $G/P$ is abelian.
Now if $G$ has a block containing only linear irreducible Brauer characters, how can we say about the group $G$?

Suppose $B \in \textrm{Bl}(G)$ and $\textrm{IBr}(B) \subseteq \textrm{LBr}(G)$.

By Theorem 2.2, it means the principal block $B_{0}$ having just linear irreducible Brauer characters.

Let $$\textrm{Ker}(B)=\bigcap_{\chi\in \textrm{Irr}(B)}\textrm{Ker}(\chi),
\textrm{Ker}(B^{*})=\bigcap_{\varphi\in \textrm{IBr}(B)}\textrm{Ker}(V_{\varphi}).$$
Then we know $$ \textrm{Ker}(B)=O_{p'}(\textrm{Ker}(\chi)),\chi \in \textrm{Irr}(B)
$$ and $$\textrm{Ker}(B^{*})/\textrm{Ker}(B)=O_{p}(G/\textrm{ker}(B))$$
If $B=B_{0}$ is the principal block, then $$ \textrm{Ker}(B_{0})=O_{p'}(G)$$ and
$$\textrm{Ker}(B_{0}^{*})/O_{p'}(G)=O_{p}(G/O_{p'}(G))$$

Suppose that $B_{0}$ contains only linear Brauer irreducible characters and $l(B_{0})>1$( If $l(B_{0})=1$, then $G$ is $p$-nilpotent). Let $1\neq \lambda \in \textrm{IBr}(B_{0})$. Now let $H=G'O^{p'}(G)$. Since $H\leq \textrm{Ker}(\lambda)$, we can see $\lambda$ as an ordinary irreducible character of $G$ by the following way:
 $$\textrm{IBr}(G/H)=\textrm{Irr}(G/H)\subseteq \textrm{Irr}(G).
 $$
 Then $$O_{p'}(\textrm{Ker}(\lambda))=O_{p'}(G).$$
 In another word, we may consider $\textrm{IBr}(B_{0})\subseteq \textrm{Irr}(B_{0})$.

Then
$$H\leq \textrm{Ker}(B^{*}_{0})=\bigcap_{\lambda\in \textrm{IBr}(B_{0})}\textrm{Ker}(V_{\lambda})=\bigcap_{\lambda\in \textrm{IBr}(B_{0})}\textrm{Ker}(\lambda).$$

Now by using discussion above, we have the following result:

\begin{theorem}Let $H=G'O^{p'}(G)$. Let $G$ has a block $B$ such that ${\rm IBr}(B)\subseteq {\rm LBr}(G)$. Then ${\rm Ker}(B_{0}^{*})=HO_{p'}(G)$.
\end{theorem}
\begin{proof} By the previous arguments, the case will be changed  to consider the principal block $B_{0}$ satisfying our condition. Since $G/H$ is a $p'$-group, ${\rm Ker}(B_{0}^{*})/HO_{p'}(G)$ must be a $p'$-group. The group ${\rm Ker}(B_{0}^{*})/HO_{p'}(G)$ is in fact a $p-$group since ${\rm Ker}(B_{0}^{*})/O_{p'}(G)$ is a $p-$group. Hence
${\rm Ker}(B_{0}^{*})/HO_{p'}(G)=1.$
\end{proof}

The following is our main result in this paper.

\begin{theorem}Let $G$ be a finite group and $p$ be a prime number. Let $H=G'O^{p'}(G)$. Then
\begin{enumerate}
\item $G$ has a $p$-block $B$ such that ${\rm Irr}(B) \subseteq {\rm Lrr}(G)$ if and only if $G$ is a $p$-nilpotent group with an abelian Sylow $p$-group;

\item $G$ has a $p$-block $B$ such that ${\rm IBr}(B)\subseteq {\rm LBr}(G)$ if and only if $O_{p'}(G)\lhd O_{p'p}(G)=HO_{p'}(G)={\rm Ker}(B_{0}^{*})\lhd O_{p'pp'}(G)=G$.

\end{enumerate}
\end{theorem}
\begin{proof} (1). If $G$ is $p$-nilpotent group with an abelian Sylow $p$-group, we just check the principal  block satisfying $\textrm{Irr}(B_{0})\subseteq \textrm{Lrr}(G)$. Since $G'\leq O_{p'}(G)=\textrm{Ker}(B_{0})\leq \textrm{Ker} \chi, \chi \in \textrm{Irr}(B_{0})$, we have $\textrm{Irr}(B_{0})\subseteq \textrm{Lrr}(G)$. Suppose a block $B$ of $G$ satisfies $\textrm{Irr}(B)\subseteq \textrm{Lrr}(G)$. Then we have $\textrm{IBr}(B)\subseteq \textrm{LBr}(G)$ and $l(B)=1$ by using the decomposition matrix and Cartan matrix of $B$. By Theorem 2.2, the principal block $B_{0}$ has  the same properties. Hence $G$ is $p$-nilpotent. From $\textrm{Irr}(B_{0}) \subseteq \textrm{Lrr}(G)$, it follows that $G'\leq \bigcap_{\chi \in \textrm{Irr}(B_{0})}\textrm{Ker}(\chi)=O_{p'}(G)$. Hence the Sylow $p$-subgroup must be abelian.

(2) The ``only if" part is from the theorem 2.3. If $HO_{p'}(G)={\rm Ker}(B_{0}^{*})$, then $\textrm{IBr}(B_{0})\subseteq \textrm{IBr}(G/H)=\textrm{Irr}(G/H)=\textrm{Lrr}(G/H)=\textrm{LBr}(G/H)=\textrm{LBr}(G)$. Hence the  ``if" part
 is proved.
\end{proof}

The following result is related to Holm and Willem's local conjecture on Brauer character degrees. Of course, our result could contains much more information of group $G$, as it connects several block invariants in an equality.

\begin{theorem}
If $G$ has a block $B$ satisfying ${\rm IBr}(B)\subseteq {\rm LBr}(G)$, then for any $B\in {\rm Bl}(G)$,
there is $$   \frac{{\rm Dim}_{F}(B)}{\abs{D_{B}}} =\sum_{\varphi\in {\rm IBr}(B)}\varphi(1)^{2}$$where $D_{B}$ denote the defect group of $B$. Hence Holm and Willems' conjecture holds for these groups.
\end{theorem}
\begin{proof}By Theorem 2.4, in $G$ there is a sequence of normal subgroups:
$$ O_{p'}(G)\unlhd O_{p'p}(G)=HO_{p'}(G)\unlhd O_{p'pp'}(G)=G$$where $H=G'O^{p'}(G)$. Notice the group $HO_{p'}(G)$ is $p$-nilpotent, hence $l(b)=1$ for any block $b\in \textrm{Bl}(HO_{p'}(G))$. Since $G/HO_{p'}(G)$ is a $p'$-group, we have $D_{B}\leq HO_{p'}(G)$. Our assertion follows from \cite{zeng}.
\end{proof}

\begin{remark} In fact, Theorem 2.5 holds for all finite groups of $p$-length $1$ which are defined in \cite{issm}.
\end{remark}

\bibliographystyle{amsplain}

\end{document}